\documentclass[reqno, 11pt]{amsart}

\usepackage{amsfonts, amsthm, amsmath, amssymb}
\usepackage{hyperref}
\hypersetup{colorlinks=false}

\usepackage{bbm}

\usepackage[margin=1in]{geometry}

\usepackage{helvet}

\RequirePackage{mathrsfs} \let\mathcal\mathscr

\usepackage{hyperref} 
\usepackage{enumerate}

\usepackage{ dsfont }

\numberwithin{equation}{section}

\newtheorem{theorem}{Theorem}[section] 
\newtheorem{lemma}[theorem]{Lemma}
\newtheorem{proposition}[theorem]{Proposition}
\newtheorem{corollary}[theorem]{Corollary}

\theoremstyle{definition}

 \newtheorem*{acknowledgements}{Acknowledgements}

 \newtheorem*{notation}{Notation}

\renewcommand{\phi}{\varphi}

\renewcommand{\leq}{\leqslant}

\renewcommand{\geq}{\geqslant}

\renewcommand{\b}{\mathbf{b}}

\renewcommand{\r}{\mathbf{r}}

\newcommand{\md}[1]{  \left(\textnormal{mod}\ #1\right)}

\newcommand{\N}{\mathbb{N}}
\newcommand{\R}{\mathbb{R}}

\renewcommand{\l}{\left}

\renewcommand{\r}{\right}
\renewcommand{\b}{\mathbf}

\renewcommand{\epsilon}{\varepsilon}

\renewcommand{\leq}{\leqslant}
\renewcommand{\geq}{\geqslant}
\renewcommand{\#}{\sharp}

\newcommand{\beq}[2]
{
\begin{equation}
\label{#1}
{#2}
\end{equation}
}

\title 
[Asymptotics of the $k$-free diffraction measure via discretisation]
{Asymptotics of the $k$-free diffraction measure via discretisation} 
 
\author{Nick Rome} 
\address{School of Mathematics
\\ University of Bristol
 \\ Bristol \\  BS8 1TW \\  UK}  
\address{
 IST Austria
\\ 
Am Campus 1
 \\ 3400 Klosterneuburg
 \\  
Austria}  
\email{nick.rome@bristol.ac.uk}

\author{Efthymios Sofos}  
\address{Mathematics Department    \\ Glasgow University
\\   Glasgow   \\ G12 8QQ  \\  UK}  
\email{efthymios.sofos@glasgow.ac.uk}

 
\date{\today}

\begin{document}

\vspace{-2.5cm}

\begin{abstract}    
We determine the  diffraction intensity 
   of the $k$-free 
integers near the origin. 
\end{abstract}  

\maketitle
 
\vspace{-0,4cm}

\setcounter{tocdepth}{1} 

\section{Introduction}   \label{s:intro}  
Point sets in Euclidean space exhibiting pure point diffraction play an important r\^ole in the theory of aperiodic order as mathematical models of quasicrystals.  The growth of the diffraction intensity 
  $Z(\epsilon)$  as $\epsilon\to 0^+$
demonstrates how stable the structure of the point set is. For instance, 
a homogeneous Poisson process displays growth $Z(\epsilon) = \epsilon$, whereas for a lattice one has $Z(\epsilon) = 0$.
Power laws are typical of aperiodically ordered sets (c.f.~\cite{statmech}), however, the only previously known example of a non-integer exponent is given by the   Thue--Morse sequence.

Recently, sets of number theoretic origin, such as the $k$-free integers, have gained attention  
as they are conjectured to be   weak model sets with extremal density.  
Baake and Coons \cite{arXiv:1804.05768}
studied the fluctuation of the density of this set by considering the scaling behaviour of the 
diffraction measure $\nu_k$, given 
by $Z_k(\epsilon) = \nu_k((0,\epsilon])/\nu_k(\{0\})$, as $\epsilon \rightarrow 0^+$.
 They used a  sieving argument to show  
\[ \lim_{\epsilon \to 0 } \frac
{\log  Z_{k} (\epsilon)}
{\log \epsilon}
=2-1/k
 .\]     
We prove that  a power law holds for $k$-free integers, thus confirming the conjectured behavior: 
 \begin{theorem}
\label
{thm:main} 
For all 
 $k>1 $, as $\varepsilon \rightarrow 0^+$ we have 
 \[
Z_k(\epsilon)=
\frac{c_k }{2 k}  
  \epsilon^{2- 1/k   }   
\l(1 +o( \epsilon^{  1/ k  }   )   \r) 
,\] where  $c_k$ is an explicit positive constant.\end{theorem}
The constant $c_k $ is given
 in \eqref{eq:foryoursafety}. It stabilises quite rapidly, specifically, 
\[ c_k
=1+ O(  1/k )
.\]  
Our proof gives an explicit error term, namely, 
$o(\epsilon^{1/k } ) $ can be replaced by 
\[\epsilon^{ 1/k  } 
\exp\{ -\gamma 
k^{-1} 
\l(\log 1/ \epsilon   \r)^{3/5}
\l(\log \log  1/ \epsilon   \r)^{-1/5}
\}  
 \]  
for some  positive  absolute     constant 
$\gamma  $.
The improvements over previous works
stem from using a  
discretisation approach, which is new in this problem
and allows the use  of  number theory estimates.

We shall see that
the Riemann hypothesis
implies 
a much stronger approximation  
of 
the diffraction intensity 
by a power law;
we are not aware
of a previous connection 
between  the Riemann hypothesis 
and aperiodic structures. 
 \begin{theorem}
\label
{thm:main rieman} 
Assume the Riemann Hypothesis.
For every $k>1 $ and 
 $\delta>0$, as $\varepsilon \rightarrow 0^+$
we have 
\[
Z_k(\epsilon)=
\frac{c_k }{2k}
  \epsilon^{2-1/k  }   
+O 
(
  \epsilon^{2- 11/35 k -\delta}  
)
.\] 
\end{theorem}

 \subsection{The discretisation approach}  
Our
method is  entirely  
different from the one used by 
Baake and Coons.
Before explaining its steps, we must note that the crucial
reason behind our
 improvements 
over the work of Baake and Coons is our  discretisation trick
and not the use of  
analytic number theory estimates. 
Indeed,   our discretisation trick followed by 
a  
sieving argument that is similar to the one of Baake and Coons, 
  would
 produce an error term 
$O(\epsilon^{1/k}) $. This 
  is plainly    weaker than our Theorem \ref{thm:main} 
but still an ample
improvement over what was known before.  

Our proof 
  has three steps:
\begin{enumerate}
\item (\textit{Discretisation})
We  
 approximate
$Z_k(\epsilon)$
by  
 $Z_k(1/N)$ for   a
certain  
 integer  $N$ in Lemma~\ref{lem:first approxim}. 
We make a slightly unusual use 
of the 
auxiliary variable   in
Lemma~\ref{lem:the new variable}:
noting  that it 
divides certain integers
allows   expressing
 $Z_k(1/N)$
as   a
sum of certain   
quantities $z_k(c)$.  
These objects 
are closer to number theory than the diffraction measure.
\item  (\textit{Analysing $z_k(c)$}) 
In  Lemma~\ref{lem:factorising}
and
Proposition~\ref{prop:final step} we study  $z_k(c)$.   
 By taking the validity of Proposition~\ref{prop:final step}
for granted we 
prove
     Theorem~\ref{thm:main}
at the end of \S \ref {s:zkc2}.  
\item (\textit{Zero-free region information})
The proof of 
Proposition~\ref{prop:final step}
is given in  \S\ref{s:summoning the demons}. It
 uses 
a result of Walfisz 
on the distribution of square-free numbers, 
whose proof hinges upon  
the zero-free region of the Riemann zeta function. 
\end{enumerate}
The 
leading constant of 
Theorem \ref{thm:main}
is analysed in 
  Section
\ref{s:lead const}.
Section
\ref {s:under rieman} gives 
the implications of 
the  Riemann hypothesis
about the diffraction measure, namely
Theorem \ref {thm:main rieman}. 
 \begin{notation}
All implied constants 
in 
the Landau/Vinogradov 
$O$-big
notation
$O(), \ll  $
are absolute.
Any further dependence on
a further  quantity $h$
 will be recorded by the use of a subscript 
$O_h ( ), \ll_h$.
The number of positive integer divisors of an
integer $n$
is denoted by
$\tau(n)$,
the M\"{o}bius function by $\mu(n)$
and the indicator function of the $k$-free integers $n$ 
by $\mu_k(n)$.
\end{notation}

\begin{acknowledgements} We thank 
Michael Baake and Michael Coons
for helpful comments that helped improve the exposition
of this work.
\end{acknowledgements}

\section{The proof of Theorem~\ref{thm:main}}
\subsection{Discretisation}
\label
{s:the auxiliary variable} 
For any $k, N\in\N$ we let 
\beq
{def:ckN}
{
\widetilde{Z}_k(N)
:=
\sum_{\substack { q\in \N   }} \mu_{k+1}(q)
\bigg
(
\prod_{\substack{ p \text{ prime}  ,  p\mid q } }
\frac{1}{(p^k -1 )^2}
\bigg
)
\#\l\{ m \in \N \cap \l [1,\frac{q}{N} \r ] :  \gcd(m,q)=1 \r\}
.} 
The function 
$\widetilde{Z}_k(N)$ 
is well-defined because 
its modulus 
is at most 
\beq
{eq:absolute convergence}
{
\sum_{\substack { q\in \N   }} \mu_{k+1}(q)
\bigg
(
\prod_{\substack{ p \text{ prime}  ,  p\mid q } }
\frac{1}{(p^k -1 )^2}
\bigg
)
q
\leq 
\prod_{p  }
\l(1+ 
\sum_{n=1}^k \frac{p^n }{(p^k -1)^2 }
\r )
\leq \prod_{p   }
\l(1+ 
 \frac{k   }{p^k -2 }
\r )
<\infty
.}
\begin{lemma}
\label
{lem:first approxim}
For any $\epsilon \in (0,1)$ let $N$ be the integer part of $1/\epsilon$.
Then 
$\widetilde{Z}_k(N+1)
\leq 
Z_k(\epsilon)
\leq 
\widetilde{Z}_k(N)
$.
\end{lemma}
\begin{proof}
The 
 work of
Baake, Moody and Pleasants~\cite{MR1778906}
gives
\[
Z_k(\epsilon)
=
\sum_{q \geq 1/\epsilon } \mu_{k+1}(q)
\sum_{\substack{ 1\leq m \leq q\epsilon \\  \gcd(m,q)=1 } } 
\prod_{p\mid q }\frac{1}{(p^k -1 )^2}
.\] 
The condition $q \geq 1/\epsilon$ is implied by the presence of the sum over
$m$ and it can therefore be omitted. The inequality $N\leq \frac{1}{\epsilon} <N+1$
shows that 
$\widetilde{Z}_k (N+1)$ equals 
\[
\sum_{q =1 }^\infty 
\mu_{k+1}(q)
\sum_{\substack{ 1\leq m \leq q/(N+1)  \\  \gcd(m,q)=1 } } 
\prod_{p\mid q }\frac{1}{(p^k -1 )^2}
\leq 
Z_k(\epsilon) 
\leq 
\sum_{q =1 }^\infty \mu_{k+1}(q)
\sum_{\substack{ 1\leq m \leq q/N  \\  \gcd(m,q)=1 } } 
\prod_{p\mid q }\frac{1}{(p^k -1 )^2}
=\widetilde{Z}_k(N)
.
\qedhere
\]
 \end{proof}

\begin{lemma}
\label
{lem:the new variable}
For any 
positive integer $N$ 
we have 
\beq
{eq:morbangel}
{
\widetilde{Z}_k(N)
=
\sum_{\substack{ c\in \N  \\ N \text{ divides } c }}
z_k(c)
,}
 where 
\[
z_k(c)
:=
\sum_{\substack{ r\in \N  \\  r \geq c   }}
\sum_{\substack{ d\in \N   }}
\mu(d) 
\mu_{k+1}(dr )
\prod_{p\mid dr   }\frac{1}{(p^k -1 )^2}
.
\]
\end{lemma}
\begin{proof}
The changes in the order of summation in the following arguments are   justified 
by the absolute convergence of the sum in~\eqref{def:ckN}, 
which is proved in~\eqref{eq:absolute convergence}. 
The expression 
$$
\sum_{\substack{ d \in \N \\ d \mid m , d \mid q   }} \mu(d) 
$$ is the indicator function of the event $\gcd(m,q)=1$.
Injecting it into~\eqref{def:ckN} yields 
\[
\widetilde{Z}_k(N)
=
\sum_{d \in \N} \mu(d) 
\sum_{\substack { q\in \N  \\  d \mid q   }} 
\mu_{k+1}(q)
 \l [ \frac{q}{d N} \r ] 
 \prod_{\substack{ p\mid  q   } }
\frac{1}{(p^k -1 )^2}
 ,\] 
where $[x]$ denotes the integer part of a real number $x$.
The integers $q$ appearing above are of the form $d r $ for some $r \in \N$, 
hence, \[
\widetilde{Z}_k(N)
=
\sum_{d \in \N} \mu(d) 
\sum_{\substack { r \in \N     }} 
\mu_{k+1}(dr )
 \l [ \frac{r}{ N} \r ] 
 \prod_{\substack{ p\mid  d r     } }
\frac{1}{(p^k -1 )^2}
.\] 
We now replace the term $[r/N]$ by 
$
\#\{ c\in \N \cap [1,r] : c \equiv 0 \md {N }  \}
$,
thus obtaining 
\[
\widetilde{Z}_k(N)
=
\sum_{\substack{  c \in \N \\  N \mid c   }}
\sum_{\substack{  r \in \N  \\ r\geq c    }}
\sum_{d \in \N} \mu(d) 
 \mu_{k+1}(dr )
  \prod_{\substack{ p\mid  d r     } }
\frac{1}{(p^k -1 )^2}
.
\qedhere
\] 
 \end{proof}

\subsection{Analysing $\mathbf{z_k(c) } $}
\label{s:zkc2}
We   express
 $z_k(c)$ via the tail of a convergent series. 
\begin{lemma}
\label
{lem:factorising}
For any 
positive integer $c $ 
we have 
\[
z_k(c)
 =
\xi_k 
\sum_{\substack{ t  \in \N \cap [c^{1/k } , \infty )   }}   
\frac{  |\mu(t) | }{t^{2k}}
\prod_{p\mid t } \frac{1}{1-  \frac{2}{p^k } } 
,\]
where 
$
\xi_k
:=
\prod_{p } 
(1 -  (p^k -1 )^{-2} 
)$.  
\end{lemma}
\begin{proof}
The integers $r$ in the definition of $z_k(c)$ 
are $(k+1)$-free,
hence 
can be written uniquely as 
$
r= \prod_{i=1}^k r_i^i
,$ where $r_i\in \N$ are square-free and coprime in pairs.
The integer $d$   in the definition of $z_k(c)
$ is square-free and therefore coprime to  
$r_k$.
Therefore, letting   
$\delta_i:=\gcd(r_i ,d)$
we infer that 
 there are unique positive  integers  
$\delta_i, s_i, (0< i < k )$, $d_0 $ such that  
\[ 
r_i= \delta_i s_i \ \ (0< i < k ), \ \ d= d_0 \prod_{i=1}^ {k-1} \delta_i. \]
 Writing $m=r_k^k \prod_{i=1}^{k-1} (\delta_i s_i )^i $     
transforms $z_k(c)$ into 
\[
\sum_{ m \geq c 
 }
\l( \prod_{p\mid m    } (p^k -1 )^{-2} \r)
 \l(  \sum_{\substack{ 
d_0 \in \N  
\\
\gcd(d_0,  
 m )=1 
  }} 
\mu(d_0  )  
\prod_{p\mid  d_0   }\frac{1}{(p^k -1 )^2}\r)
\sum_{\substack{  
r_k \in \N  , \boldsymbol \delta \in \N^{k-1}, \b s \in \N^{k-1} 
\\
m=r_k ^k \prod_{i=1}^{k-1} (s_i \delta_ i )^i   
}}
\mu(\delta_1)  \cdots \mu(\delta_{k-1} )
,
\]
where the sum over $r_k , \boldsymbol \delta , \b s$ is subject to the conditions 
$\gcd( s_i \delta_i , s_j \delta_j ) =1 $ for all $i\neq j $
and  $\gcd( r_k  \delta_ i , s_i )  =1$ for all $i\neq k $. 
One can see that the sum over 
$r_k , \boldsymbol \delta , \b s$ forms  a multiplicative function of  $m $
and looking at its values at prime powers
makes  clear that it 
is   the indicator function of integers of the form 
$m=t^k $ with $t$ square-free. Indeed, for $1\leq j  < k $ we have 
\[\sum_{\substack{  
r_k \in \N  , \boldsymbol \delta \in \N^{k-1}, \b s \in \N^{k-1} 
\\
p^j=r_k ^k \prod_{i=1}^{k-1} (s_i \delta_ i )^i   
}}
\mu(\delta_1)  \cdots \mu(\delta_{k-1} )
=\sum_{\substack{   s_j, \delta_j \in \N 
\\
p^j = (s_j \delta_ j )^j  
}}
\mu(\delta_j)=\sum_{\delta_j \mid p}    
\mu(\delta_j)=0 
\] since all other variables in the   sum must equal  $1$.
We thus obtain 
   \[z_k(c)  
= \sum_{ t \geq c^{1/k }   }
|\mu(t) |  \l( \prod_{p\mid  t     } (p^k -1 )^{-2} \r) 
   \sum_{\substack{ 
d_0 \in \N  
\\
\gcd(d_0,  t )=1 
  }} 
\mu(d_0  )  
\prod_{p\mid  d_0   } (p^k -1 )^{-2}  , \]
The proof concludes by writing the sum over $d_0 $ as an Euler product
and using $t^{-2k}=\prod_{p\mid t}p^{-2k}$.\end{proof}
 
\begin{proposition}
\label
{prop:final step}
There exists a 
constant 
$\gamma >0$
such that  
for all  $k>1 $
and 
$u \geq 1 $
we have \[ \sum_{\substack{  
1\leq 
  t \leq u  
}}
|\mu( t )|      \prod_{p\mid t } \frac{1}{1-2p^{-k}} 
= \gamma _k u
 +O_k\l(
\frac{  
u^{1/2} } { \exp\l( \gamma
(\log  u  )^{3/5}
(\log \log u )^{-1/5}
\r)   }
\r ) 
,
\]  
where 
 the implied constant depends at most on $k$
and 
\[
\gamma_k :=  \frac{1 }{\zeta(2) }
 \prod_p \l( 1+\frac{   2    }{(p+1)  (p^k -2)}\r )
 .\]   \end{proposition} 
We conclude this 
section by 
deducing  
Theorem~\ref{thm:main}
from
Proposition~\ref{prop:final step}.  
Define 
\beq
{eq:foryoursafety} 
{
c_k := \frac{2k}{(2k-1) }
\frac{\zeta(2-1/k) }{\zeta(2) }
\zeta(k)^2
\prod_{\substack{ p \\ }}  \l(1-\frac{2p}{(p+1) p^{ k} }\r) 
.}
\begin{proof}
[Proof of Theorem~\ref{thm:main}]
  Lemma \ref{lem:factorising},
Proposition \ref{prop:final step} and 
Abel's summation formula give \[
\frac{z_k(c)}{\xi_k}
=\frac{ \gamma _k }{ (2k-1 ) } 
\frac{ 1 }{ c^{2-1/k }   } 
 +O_k\l(
\frac{  c^{-2+1/(2k)} } 
{ 
\exp\l( \gamma k^{-1 } 
(\log  c  )^{3/5}
(\log \log c )^{-1/5}
\r)   
}
\r ) .\] 
Feeding this 
into
Lemma  \ref{lem:the new variable} 
produces
\[\widetilde{Z}_k(N)
=
\sum_{\substack{ b \in \N   }}    z_k(N b)
=
  \frac{ \gamma_k  \xi_k }{ (2k  - 1  )    }
 \frac{ \zeta\l(2-\frac{1}{k} \r)}{N^{2-\frac{1}{k}  }   }
+O_k\l(
 \frac{ 
 N^{-2+1/(2k)} 
} { \exp\l( \gamma 
k^{-1} 
(\log  N  )^{3/5}
(\log \log N )^{-1/5}
\r)   }
\r)
.\] The leading constant can be turned into the form of Theorem~\ref{thm:main} by noting that 
\[ 
\frac{ \gamma_k   \xi_k   }{(2k-1) } \zeta(2-1/k)
  =\frac{c_k}{2k }
.\]
Finally, 
invoking Lemma~\ref{lem:first approxim}
concludes the proof 
because the inequality $N \leq \frac{1}{\epsilon} <1+N$ implies 
that both
$
(N+1) ^{-2+\frac{1}{k}  }   
$
and 
$  N^{-2+ \frac{1}{k}  }     $ are
$
\epsilon^{2-\frac{1}{k}  } 
+
O_k(\epsilon^{3-\frac{1}{k}  }   )
$. 
\end{proof}
 \subsection{Zero-free region information}
\label
{s:summoning the demons}
We now
prove
Proposition~\ref{prop:final step}
by using
the following result, which 
is based 
on
the best known 
zero-free region for the Riemann zeta function.
\begin{lemma}
[Walfisz,~\cite{MR0220685}]
\label
{lem:walf} 
 There exists an absolute constant $
\gamma_0
>0$
such that 
 \[
\sum_{n \in \N \cap [1,x] }
\mu(n)^2
=
\frac{x}{\zeta(2) }
+O
\l(
x^{\frac{1}{2}}
\exp\l(-\gamma_0
(\log x )^{3/5}
(\log \log x)^{-1/5}
\r)
\r)
.\]
\end{lemma}
We shall later need a stronger version of Lemma \ref{lem:walf}. 
\begin{corollary}
\label
{cor:walf} 
There exists an absolute constant $\gamma'>0$
such that  
 for every  $a\in \N, x\geq 1 $ 
we have 
\[
\sum_{\substack {   n \in \N \cap [1,x] \\ \gcd(n,a)=1 } }  \mu(n)^2
=
\l(\prod_{p\mid a } \l(1+\frac{1}{p}\r)^{-1 }\r)
\frac{x}{\zeta(2) }
+O
\l(
\tau(a)^3
x^{\frac{1}{2}}
\exp\l(-\gamma'
(\log x )^{3/5}
(\log \log x)^{-1/5}
\r)
\r)
,\]
where the implied constant 
is absolute.
\end{corollary}
\begin{proof}
The Dirichlet series of 
$\mathds 1 _{\gcd(a,n)=1}(n)
\mu(n)^2 $ is 
\[
\sum_{\substack {   n  =1 \\ \gcd(n,a)=1 } }^\infty
\frac{\mu(n)^2}{n^s }
  =
\prod_{p  } \l(1 +\frac{1}{p^s } \r )
\prod_{p\mid a } 
\l(1 +\frac{1}{p^s } \r )^{-1} 
=
\prod_{p  } \l(1 +\frac{1}{p^s } \r )
\prod_{p\mid a } 
\l(1 +\sum_{k=1}^\infty
\frac{(-1)^k }{p^{ks} } \r )
.
\] This is the product of the Dirichlet series 
of $\mu(n)^2 $ 
by  the Dirichlet series of the multiplicative 
 function $g_a(n)$, where 
\[
g_a(n):=
\mathds 1 _{p \mid n \Rightarrow p\mid a  }(n)
(-1)^{\Omega(n)
}
\]
and $\Omega(n)$ is the number of prime divisors of $n$ 
counted with multiplicity.
We get 
\[\mathds 1 _{\gcd(a,n)=1}(n)
\mu(n)^2
=
(\mu^2 \ast g_a)(n)
=\sum_{\substack {  c,d \in \N \\ cd =n  }}
g_a(c)
\mu(d)^2
,\] 
where  
$\ast$ is the Dirichlet convolution.
Hence,
we 
can write 
\[
\sum_{\substack {   n \in \N \cap [1,x] \\ \gcd(n,a)=1 } }  \mu(n)^2
=
\sum_{1\leq c \leq x }   g_a(c)
\sum_{1\leq d \leq x/c}   \mu(d)^2
.
\]
Let $Y:=x^{7/10}$.
The terms with 
 $
Y<
c\leq x$
contribute at most
 \[x
\sum_{c>
Y
}
\frac{ |g_a(c) | }{c} 
\leq 
\frac{x}{Y^{3/4}}
\sum_{c>
Y  
}
\frac{ |g_a(c) | }{c^{1/4} } 
\leq \frac{x}{Y^{3/4}}
\prod_{p\mid a  } \l(1+\sum_{k=1}^\infty   \frac{1}{p^{k/4}}   \r)
\leq
 \frac{x}{Y^{3/4}}
\l (\prod_{p\mid a } 8 \r )
\leq
 \frac{x  \tau(a)^3}{Y^{3/4}}
,\] which equals  $
 \tau(a)^3
x^{\frac{19}{40 } }
$.
By 
Lemma~\ref{lem:walf},
the terms with $c\leq Y$ 
contribute
\[
\frac{x}{\zeta(2)  }
\sum_{1\leq c \leq Y }  \frac{  g_a(c) }{c }
  +O
\l(
x^{\frac{1}{2}}
\sum_{1\leq c \leq Y }  \frac{  |g_a(c) |}{c^{\frac{1}{2}}  }
\exp\l(-\gamma_0
(\log x/c )^{3/5}
(\log \log x/c)^{-1/5}
\r)
\r)
.\]
Note that $x/c \geq x^{3/10}$, therefore, 
\[
(\log x/c )^{3/5}
(\log \log x/c)^{-1/5}
\geq 
5^{-3/5}
(\log x )^{3/5}
(\log \log x/c)^{-1/5}
.\]
Letting
 $\gamma'
:=5^{-3/5}\gamma_0$
 we infer that 
the error term contribution
is 
\[
\ll 
x^{\frac{1}{2}}
\exp\l(-\gamma
(\log x)^{3/5}
(\log \log x)^{-1/5}
\r) 
\sum_{1\leq c \leq Y }  \frac{  |g_a(c) |}{c^{\frac{1}{2}}  }
\ll 
\tau(a)^3
x^{\frac{1}{2}}
\exp\l(-\gamma'
(\log x)^{3/5}
(\log \log x)^{-1/5}
\r) 
.\]
To complete the summation over $c>Y$
we use the estimate  
 \[
x\sum_{  c >  Y }  \frac{|  g_a(c) |}{c }
\ll 
 \tau(a)^3
x^{\frac{19}{40 } }
 \] that was 
 proved   earlier in this proof. Finally, 
the proof is concluded by noting that 
 \[
\sum_{c \in \N } \frac{g_a(c)  }{ c }
=
\prod_{p\mid a } \l(1+\frac{1}{p}\r)^{-1 }
.
\qedhere
\]
\end{proof}
The following result 
is a generalisation of 
Lemma~\ref{lem:walf} 
and its proof uses 
Corollary~\ref{cor:walf}.
\begin{lemma}
\label
{lem:the last lemma}
Let   
$\delta:\N \to \R$ be a multiplicative function 
with  $|  \delta(p)   |
\leq  \frac{4}{p^2} $    
for every  prime   $p$.
There exists a
positive absolute constant $\gamma$ such that for all $u\geq 1 $ we have  
\[
\sum_{\substack { 1\leq m \leq u   }} 
\mu(m)^2 \prod_{p\mid m  } (1+\delta(p) ) 
 =
\l(
\prod_p \l( 1+\frac{\delta(p)   }{p+1 }\r )
\r)\frac{u}{\zeta(2) } +O\l(
\frac{  
u^{1/2} } { \exp\l( \gamma
(\log  u  )^{3/5}
(\log \log u )^{-1/5}
\r)   }
\r )
,\]
 where the 
implied constant is absolute.
\end{lemma}
\begin{proof} Switching the order of summation, 
the sum in the lemma becomes 
\[
\sum_{\substack{ d \leq u   } } \delta(d)  
\sum_{\substack { 1\leq m \leq u\\   d\mid m   }} 
\mu(m)^2 
=
\sum_{\substack{ d \leq u   } } \delta(d)   \mu( d  )^2 
\sum_{\substack { 1\leq  m' \leq u/d\\  \gcd( m', d   )=1   }} 
\mu(  m' )^2 
.\]
The contribution of 
$d>u^{3/4}$
is admissible,
since it is at most 
\[
\ll 
\sum_{\substack{ d > u^{3/4}   } } \delta(d)   \mu( d  )^2  \frac{u}{d}
\leq   
\sum_{\substack{ d > u^{3/4}   } }  
\frac{ \tau(d)^2 }{d^2}
\frac{  u  }{d  }
\ll   u^{-1/4}
 \] due to $\mu(d)^2 4^{\#\{p\mid d \} }    \leq \tau(d)^2$ and 
 the divisor bound $ \tau(d ) \ll_\epsilon d^\epsilon$ for all $\epsilon>0$.  
To the remaining range, $1\leq d \leq u^{3/4}$,
we apply 
Corollary~\ref{cor:walf} 
with $x=u/d  $ and $ a=d  $.
It gives 
\[
\sum_{\substack{ d \leq u^{3/4}
  } } \delta(d)   \mu( d  )^2 
\l(
\frac{u }{ d \zeta(2) }
\l(\prod_{p\mid d   } \l(1+\frac{1}{p}\r)^{-1 }\r)
+O
\l(
\frac{ 
\tau(d)^3 
u^{\frac{1}{2}} } {d^{\frac{1}{2}}
\exp\l(\gamma'
(\log \frac{u}{d}  )^{3/5}
(\log \log \frac{u}{d } )^{-1/5}
\r)
}
\r)
 \r)
.\]
Using 
$d \leq u^{3/4}$
we see that $\log \frac{u}{d} \geq \frac{1}{4} \log u $, hence 
the error term is 
\[
\frac{ 
u^{\frac{1}{2}} } { \exp\l( \gamma 
(\log  u  )^{3/5}
(\log \log u )^{-1/5}
\r)   }
\sum_{\substack{ d \leq u^{3/4}
  } }  
\frac{\tau(d)^3  } {d^{\frac{3}{2}}}
\ll 
\frac{  
u^{\frac{1}{2}} } { \exp\l( \gamma 
(\log  u  )^{3/5}
(\log \log u )^{-1/5}
\r)   }
 \] for
 some positive 
absolute constant $\gamma $.
Finally, we complete the summation in the main term: 
\[
\sum_{\substack{ d   >    u ^{3/4}   } } 
 \frac{  \delta(d)   \mu( d  )^2   }{ d  }
 \prod_{p\mid d   } \l(1+\frac{1}{p}\r)^{-1 }
\leq   
\sum_{\substack{ d   >    u ^{3/4}    } } 
 \frac{   \tau(d)  ^2  }{ d^3   }
\ll   u^{-3/4}.
\qedhere
\]
  \end{proof}

\begin{proof}
[Proof of Proposition~\ref{prop:final step}]    This follows from applying  
Lemma~\ref{lem:the last lemma}
with 
\[
\delta(p)   =   -1+ \frac{1}{1-2p^{-k } }=\frac{2}{p^k -2 } \leq \frac{4}{p^2}
.  
\qedhere
\]
\end{proof}

\section{Analysis of the leading constant}
\label
{s:lead const}
We   analyse the   behavior 
of the leading  constant 
 in Theorem~\ref{thm:main}.
\begin{theorem}
\label
{prop:asymptot as k to infty}
The following holds for all $k>1 $
and with an absolute implied 
constant,
\[
c_k = 1 +O\l( \frac{1}{k  } \r)
. \]  
\end{theorem}
\begin{proof}  
We note that $2p/(p+1) \leq p^{1/2} $, hence,  \[ 
\prod_p \l(1-\frac{2p}{(p+1) p^{ k} }\r) 
\geq \zeta(k-1/2)^{-1} .\] 
Using $\zeta(\sigma)=1+O(2^{-\sigma})$ for $\sigma>3/2$
and \eqref {eq:foryoursafety} 
yields 
\[c_k =
\frac{2k}{(2k-1) }
\frac{\zeta(2-1/k) }{\zeta(2) } 
(1+O(2^{- k  })  ) 
.\] We conclude the proof by 
 using
the bound 
$\max   \l\{  |\zeta'(\sigma)|: \frac{3}{2}   \leq  \sigma<  2 \r\}  =O(1)$
to infer that
\[
\zeta\l(2-\frac{1}{k}\r)
=\zeta(2)
+O\l(\frac{1}{k}\r).
\qedhere
\]
\end{proof}

\section{Approximations via the Riemann Hypothesis}
\label
{s:under rieman}
In this section we prove Theorem \ref{thm:main rieman}. 
The main   input 
is the following result.
 \begin{lemma}
[Liu,~\cite{MR3412720}]
\label
{lem:liu}
 Assume the Riemann Hypothesis.
Then for every fixed 
 $\delta>0$ we have 
\[
\sum_{n \in \N \cap [1,x] }
\mu(n)^2
=
\frac{x}{\zeta(2) }
+O_\delta
\l(
x^{\frac{11}{35}+\delta}
\r)
.\]
\end{lemma}
This result 
uses 
van der Corput’s
method for estimating exponential sums.

\begin{corollary}
\label
{cor:liu}
Assume the Riemann Hypothesis
and let 
 $\delta>0$
be arbitrary and 
fixed.
Then for every  $a\in \N$ and $x \geq 1 $
we have 
\[
\sum_{\substack {   n \in \N \cap [1,x] \\ \gcd(n,a)=1 } }
\mu(n)^2
=
\frac{x}{\zeta(2) }
\l(\prod_{p\mid a } \l(1+\frac{1}{p}\r)^{-1 }\r)
+O_\delta
\l(
\tau(a)^3
x^{
\frac{11}{35}+\delta
}
\r)
,\]
where the implied constant 
depends at most on $\delta$.
\end{corollary}
\begin{proof}
We make use of the function 
$g_a(n)$ that is defined in the proof of Corollary~\ref{cor:walf}. Thus
the sum in our corollary equals 
\[
\sum_{1\leq c \leq x }   g_a(c)
\sum_{1\leq d \leq x/c}   \mu(d)^2
=
\frac{x}{\zeta(2)}
\sum_{1\leq c \leq x }   
\frac{ g_a(c) }{c}
+O_\epsilon\l( 
x^{{\frac{11}{35}+\epsilon
}}
 \sum_{1\leq c \leq x }   
\frac{ |g_a(c) | }{c^{\frac{11}{35}+\epsilon
}} \r )
,\]
where a use of Lemma~\ref{lem:liu}
has been made.
The bound $ | g_a(c) | \leq \mathds 1_{p\mid c \Rightarrow p\mid a }(c)$
shows that 
the error term is 
\[
\ll 
\prod_{p\mid a } 
\l(1 +\sum_{k=1}^\infty
\frac{1}{p^{{\frac{11}{35}k+\epsilon k
}
}} 
\r)
\leq 
\prod_{p\mid a } \l(1 +\sum_{k=1}^\infty
\frac{1}{2^{{\frac{11}{35}k 
}
}} 
\r)
\leq 8^{\omega(a)} 
\leq \tau(a)^3
.\]
The same bound yields 
\[
\sum_{c>x } \frac{|g_a(c)|}{c}
\leq 
\sum_{\substack{ c\in \N \\ p\mid c \Rightarrow p\mid a   }}
\l(\frac{c}{x}
\r)^{\frac{24}{35 }}
\frac{1}{c}
\leq 
x^{-\frac{24}{35}}
\prod_{p\mid a } 
\l(  1+\sum_{k=1}^\infty \frac{1 }{p^{\frac{11}{35}   }} 
 \r)
\leq 
x^{-\frac{24}{35}}
\prod_{p\mid a } 
\l(  1+\sum_{k=1}^\infty \frac{1 }{2^{\frac{11}{35}   }} 
 \r)
\leq 
\frac{\tau(a)^3
}{x^{\frac{24}{35}}
}
.
\qedhere
\]  
\end{proof}
\begin{lemma}
\label
{lem:the icecream lemma}
Assume the Riemann Hypothesis
and let 
 $\delta>0$
be arbitrary and 
fixed. 
Let   
$\delta:\N \to \R$ be a multiplicative function 
with  $|  \delta(p)   |
\leq  \frac{4}{p^2} $    
for every  prime   $p$.
 Then for all $u \geq 1 $  
we have 
\begin{align*} 
\sum_{\substack { 1\leq m \leq u   }} 
\mu(m)^2 \prod_{p\mid m  } (1+\delta(p) ) 
&=\l(
\prod_p \l( 1+\frac{\delta(p)   }{p+1 }\r )
\r)\frac{u}{\zeta(2) } 
+O_\delta
\l( 
u^{\frac{11}{35}+\delta}  
\r )
,\end{align*}
 where the 
implied constant depends at most on $\delta$.
\end{lemma}
\begin{proof} 
As in the proof of Lemma~\ref{lem:the last lemma}
we see that the sum in our lemma is 
\[
\sum_{\substack{ d \leq u   } } \delta (d)   \mu( d  )^2 
\sum_{\substack { 1\leq  m' \leq u/d\\  \gcd( m', d    )=1   }} 
\mu(  m' )^2 
,\] which, by
 Corollary~\ref{cor:liu},
is
\[\sum_{\substack{ d \leq u      } } \delta (d)   \mu( d  )^2 
\l(
\frac{u}{d\zeta(2) }
\l(
\prod_{p\mid d    } \l(1+\frac{1}{p}\r)^{-1 }\r)
+O_\delta
\l(
\tau(d)^3 
\frac
{u^{\frac{11}{35}+\delta}  }
{d^{\frac{11}{35}+\delta} } 
\r)
\r)
.\]
The main term above matches the main term in our lemma
up to a quantity that has modulus 
\[
\ll u 
 \sum_{\substack{ d > u    } } 
 \frac{
\delta (d)   \mu( d  )^2 
}{d  }
\ll
u
 \sum_{d>u}
\tau(d)^2
d^{-3} \ll 1.
 \]
The error term contribution is 
\[
\ll_\delta
{u^{\frac{11}{35}+\delta}  }
 \sum_{\substack{ d \leq u   } } 
\frac
{ 
\psi(d)   \mu( d  )^2 
  \tau(d)^3  }
{d^{\frac{11}{35}+\delta} } 
  \ll_\delta
 {u^{\frac{11}{35}+\delta}  }
\prod_{p} \l(1+\frac{32 }{p^2}\r) \ll_\delta {
u^{\frac{11}{35}+\delta}  }
 .
\qedhere
\]
\end{proof} The proof of the next lemma
follows directly from Lemma~\ref{lem:the icecream lemma}.
\begin{lemma}
\label
{lem:asymptot lambd under rh}
Assume the Riemann Hypothesis
and let 
 $\delta>0$
be arbitrary and 
fixed.
Then 
for all  $k>1 $
and 
$u \geq 1 $
we have \[ \sum_{\substack{  
1\leq 
  t \leq u  
}}
|\mu( t )|      \prod_{p\mid t } \frac{1}{1-2p^{-k}} 
= \gamma _k u
+O_{k,\delta}
( u^{\frac{11}{35}+\delta}    ) 
,
\]   
where
 the implied constant depends
at most on $\delta$ and  $k$.
\end{lemma}The proof of Theorem~\ref
{thm:main rieman} 
is now concluded  as that of 
Theorem~\ref{thm:main}
 by replacing the use of 
Proposition~\ref{prop:final step}
by Lemma \ref{lem:asymptot lambd under rh}.


\begin{thebibliography}{99}
 


\bibitem
{arXiv:1804.05768}
M. Baake and M. Coons,
Scaling of the diffraction measure of {$k$}-free integers near the origin.
{\em Michigan Math. J.,} in print (2020).
 

\bibitem
{MR1778906} 
M. Baake and R. V. Moody and P. A. B. Pleasants, 
 Diffraction from visible lattice points and {$k$}-th power free
              integers.{\em Discrete Math.} {\bf 221} (2000), 3--42.
 
 
  \bibitem
{MR3136260}
M. Baake and U. Grimm,
 {\em Aperiodic order. {V}ol. 1:
A Mathematical Invitation},
With a foreword by Roger Penrose,
Cambridge University Press, 
{\bf 149},
Cambridge,
2013.

\bibitem
{statmech} 
\bysame,
 Scaling of diffraction intensities near the origin: some rigorous results.
\textit{ J. Stat. Mech. Theory Exp.} 2019, 054003.
 

 
  \bibitem
{MR3597030} 
M. Baake and C. Huck and N. Strungaru,
On weak model sets of extremal density.
{\em Indag. Math.} {\bf 28} (2017), 3--31.

   



\bibitem
{MR3412720} 
H. Q. Liu,
On the distribution of squarefree numbers.
{\em J. Number Theory} {\bf 159} (2016), 202--222.






\bibitem
{MR0220685}
A. Walfisz,
 {\em Weylsche {E}xponentialsummen in der neueren {Z}ahlentheorie}.
VEB Deutscher Verlag der Wissenschaften, Berlin,
1963.
 

 \end{thebibliography}
\end{document}